\documentclass[11pt,a4paper]{amsart}
\usepackage{amsfonts,amsmath,amssymb,amsthm}
\usepackage{times}
\usepackage{color}
\usepackage{capt-of}
\usepackage{pstricks}
\usepackage{xspace}

\numberwithin{equation}{section}

\renewcommand{\epsilon}{\varepsilon}

\usepackage[latin1]{inputenc}  
\usepackage[T1]{fontenc}       

\usepackage{amsmath}
\usepackage{amsthm}
\usepackage{amssymb}
\usepackage{amsfonts}
\usepackage{textcomp}

\usepackage{dsfont}

\usepackage{graphicx}
\usepackage{latexsym}
\usepackage{stmaryrd}
\usepackage{amsthm, amsbsy, upref, nicefrac, amsrefs}
\usepackage{bbm}
\usepackage{ifpdf}
\usepackage{hyperref}
\usepackage[intoc]{nomencl}
\usepackage{xspace}
\makenomenclature

\usepackage{pst-all}
\usepackage{color}
\usepackage{epsfig}

\newtheorem{qst}{Question}
\newtheorem{prop}{Proposition}

\DeclareMathOperator{\re}{\mathrm{Re}}

\DeclareMathOperator{\tr}{\mathrm{tr}}
\DeclareMathOperator{\sgn}{\mathrm{sgn}}

\DeclareMathOperator{\Div}{\mathrm{div}}

\newcommand{\e}{\mathrm{e}}
\newcommand{\ii}{\mathrm{i}}

\newcommand{\cp}{compare }

\newcommand{\eg}{for example }
\newcommand{\ie}{\mbox{i.\,e.}\xspace}

\newcommand{\resp}{\mbox{respectively}\xspace}
\newcommand{\Wlog}{\mbox{Without loss of generality}\xspace}
\newcommand{\fa}{\mathbf}
\newcommand{\one}{\fa{\mathbbm{1}}}
\newcommand{\R}{\mathbb{R}}
\newcommand{\N}{\mathbb{N}}
\newcommand{\D}{\mathcal{D}}
\newcommand {\cont}{\mathcal{C}}
\newcommand \pn[1]{\emph{#1}}
\newtheorem{definition}{Definition}
\newtheorem{theorem}{Theorem}

\newtheorem{alg}{Algorithm}
\newtheorem{remark}{Remark}

\providecommand{\norm}[1]{\lVert#1\rVert}
\providecommand{\ker}[1]{\mbox{Ker}}
\providecommand{\sgn}[1]{\mbox{sgn}}
\providecommand{\Id}[1]{\mbox{Id}}
\providecommand{\ran}[1]{\mbox{ran}}
\providecommand{\deg}[1]{\mbox{deg}}
\providecommand{\diag}[1]{\mbox{diag}}

\providecommand{\det}[1]{\mbox{det}}
\providecommand{\R}{\mathds{R}}
\providecommand{\C}{\mathds{C}}
\providecommand{\Z}[1]{\mathds{Z}}
\providecommand{\N}[1]{\mathds{N}}
\providecommand{\K}[1]{\mathds{K}}
\title{On the solvability of some partial differential inequality}
\author[M. Himmel]{Martin Himmel}
\address{M.~Himmel, FB 08 - Institut f\"{u}r Mathematik,
Johannes Gutenberg-Universit\"{a}t Mainz,\newline
Staudinger Weg 9,
D-55099 Mainz,
Germany}
\email{himmel@mathematik.uni-mainz.de}
\begin{document}
  %\maketitle
%  \secnumdeph(chapter)
%  \setcounter{tocdepth}{1}
  %wenn vorwort als kapitel 0 eingefügt werden soll
  %\setcounter{chapter}{-1}
\begin{abstract}
The \pn{Dulac} criterion is a classical method to rule out the existence of periodic solutions in planar differential equations.
In this paper the applicability and therefore reversibility of this criterion is under consideration.
\end{abstract}
\maketitle

\section{Introduction and Motivation}
Let  $X: \D \to \R^2$ be a smooth vector-valued function with components $P$ and $Q$, $X=(P,Q)$,
defined on some planar domain $\D \subseteq \R^2$.
Now, consider the system of two ordinary differential equations
\begin{equation}
%\dot{x}=
\frac{dx}{dt}=P(x,y),\qquad
%=\dot{y}
\frac{dy}{dt}=Q(x,y).
\label{eq:planarODE}
\end{equation}
Such systems appear frequently in applications, \eg in electrical engineering, physics, biology and many others,
but they have also become a field of mathematical interest on their own. 
For simplicity, one can think of $P$ and $Q$ either just as polynomials of two real variables $x$ and $y$,
or as smooth or even analytical functions given by some power series that converges in $\D$,
\ie, 
$X$ is of class $\cont^r$ with $r\in \{1,2, \ldots \infty, \omega\}$.
Geometrically speaking, the solutions of the system \eqref{eq:planarODE} are smooth curves in the plane that are tangential to the function $X$ at each point $z=(x,y)\in \D$.
If the components of $X$ are polynomials in $x$ and $y$ of degree one
then
\begin{equation*}
X(x,y)=
\begin{pmatrix}
ax+by\\
cx+dy
\end{pmatrix}
%\label{eq:}
\end{equation*}
with $a,b,c,d\in \R$ being real constants,
and the analysis of system \eqref{eq:planarODE} is easy and
the solution curves of system \eqref{eq:planarODE} are given explicitly in terms of the matrix exponential function $\e^A:=\sum_{k=0}^\infty{\frac{1}{k!}A^k}$.
The situation changes drastically
if $X$ is at least quadratic,
\ie,
if $P$ and/or $Q$ are polynomials of degree two or higher,
$P,Q \in \R_n[x,y]$, $n \geq 2$.
Many things are known about quadratic differential equations, see \cites{MR1313826,MR936240,MR757291},
but they are still a broad area of research. 
Quite a few people dedicated their whole life to the investigation 
of quadratic planar differential equations. 
The interest to this field is related to the $16$-th Hilbert problem, which is still unsolved even in the quadratic case.

\subsection{Sixteenth Hilbert problem}
Let $R_d[x,y]$ denote the space given by polynomials in $x$ and $y$ of degree at most $d$,
and consider a planar system \eqref{eq:planarODE} with polynomial right-hand-side $X \in \R^2_d[x,y]$.
The $16$-th \pn{Hilbert} problem consists of two parts,
the first is a purely algebraic problem and deals with questions related to the topology of algebraic curves and surfaces. 
In this paper we are mainly concerned about the second part of the $16$-th \pn{Hilbert} problem
which deals with the curves
that arise as solutions of the planar differential equation \eqref{eq:planarODE} with polynomial vector field $X$.
More precisely, in the second part of his problem \pn{Hilbert} asks for the
maximal number and relative position of the isolated closed orbits, called limit cycles
\footnote{During history the term limit cycle was used for a stable isolated closed orbit or just for closed orbit. Nowadays and in this paper, a limit cycle is a closed orbit that is isolated in the set of all periodic orbits of some differential equation.},
the differential system \eqref{eq:planarODE} can have at most.
It is even difficult to understand why a fixed planar differential equations with polynomial vector field can only have a finite number of limit cycles.
A proof of this fact is due to \pn{Ilyashenko} in $1991$ \cite{MR1133882}, 
who actually corrected a very complicated and long proof due to \pn{Dulac} \cite{MR1504823}, a student of \pn{Poincaré}.
\pn{Écalle} et al. obtained independently a proof for this theorem \cite{MR889742}. 
Note that this does not imply the existence of an upper bound for the \textit{maximal number of limit cycles} $H(d)$
which a planar polynomial system of degree $d$ can have.
One calls $H(d)$ the $d$-th \pn{Hilbert} number.
Linear vector fields have no limit cycles;
hence $H(1) = 0$.
Quadratic systems can actually have four limit cycles
and some people believe that this is the maximal number of limit cycles
a quadratic system can have,  
but it is still unknown whether or not $H(2)$ is a finite number.
Usually, the first part of the $16$-th Hilbert problem
is studied by researchers in real algebraic geometry, while the second
part is considered by mathematicians working in dynamical systems or
differential equations. \pn{Hilbert} also pointed out that there exist possibly
connections between these two parts.
See \cite{MR1898209} and \cite{MR2072639} for the original paper in Russian by \pn{Ilyashenko}
and more recent \cite{OnThe16HilbertProblemLlibre2012} for a survey about the second part of the $16$-th \pn{Hilbert} problem.

\subsection{\pn{Dulac} criterion}\label{sec:DulacCriterion}
In $2008$ the author was confronted with the analysis of some quadratic polynomial differential equation.
In literature there existed a quite technical proof showing the uniqueness of a limit cycle for this system.
Then, fortunately, he was able to obtain the same result by applying the well-known

\begin{theorem}[\pn{Dulac} criterion \cite{MR1504823}]\label{theo:DulacCriterion}
Let $\Omega \subseteq \R^2$ be a simply connected region, $X:=(P,Q) \in \cont^1(\Omega, \R^2)$ a smooth vector field
and $B\in \cont^1(\Omega, \R)$ a smooth real-valued function - 
such that
the partial differential inequality
\begin{equation}
\Div(B X):=\frac{\partial(BP)}{\partial{x}}+ \frac{\partial(BQ)}{\partial{y}}>0 
\label{eq:DulacCondition}
\end{equation}
is satisfied in  $\Omega$.
Then the planar ordinary differential equation \eqref{eq:planarODE} with $X$ as right-hand-side 
does not posses any periodic solution that is fully contained in $\Omega$.
\end{theorem}

\begin{remark}
\begin{enumerate}
	\item 
The proof of theorem \ref{theo:DulacCriterion} is indirect: 
One assumes the existence of a closed orbit of \eqref{eq:planarODE} in $\Omega$ and applies the divergence theorem of \pn{Gau{\ss}}.
This gives immediately a contradiction.
\item
In $2009$ the author was able to weaken the assumptions made on the function $B$, see \cite{Himmel:2009:OTE}. 
Essentially, the statement of theorem \ref{theo:DulacCriterion} still holds
if the function $B$ is only weakly differentiable of first degree 
and equation \eqref{eq:DulacCondition} holds only almost everywhere in $\Omega$, \cp definition \ref{def:DulacFunction}.
\item
If the domain $\Omega \subseteq \R^2$ referred to in theorem \ref{theo:DulacCriterion} is not simply connected, 
but $p$-connected for some $p \in \N$, $p \geq 2$,
then there can be at most $p-1$ closed orbits fully contained in $\Omega$.
\item
The \pn{Dulac} criterion is a generalization of the \pn{Bendixson} criterion \cite{MR0350126}.
In fact, the \pn{Bendixson} criterion follows from theorem \ref{theo:DulacCriterion}
if we choose $B=\one$, the function that is constant $1$ for all $z\in \Omega$.
\end{enumerate}
\end{remark}

\begin{definition}[\pn{Dulac} function]\label{def:DulacFunction}
The multiplier $B\in W^{1,p}(\Omega,\R)$ is called \pn{Dulac} function of the planar dynamical system \eqref{eq:planarODE} in $\Omega \subseteq \D$
if and only if
there is a real-valued continuous function $g>0$ having positive sign in $\Omega$ except on a set of \pn{Lebesgue} measure zero 
such that equation
\begin{equation}
\Div(B  X)=B\cdot \Div{X}+
\langle{\nabla{B},X}\rangle=g
\label{eq:defDulacFunction}
\end{equation}
holds in $\Omega$.
Here $W^{1,p}(\Omega,\R)$ denotes the \pn{Sobolev} space of $L^p$ functions that are weakly differentiable of first degree
and with weak derivative in $L^p$.
Vice versa, we call any function $B$ of class $W^{1,p}$ satisfying \eqref{eq:defDulacFunction} a $g$-\pn{Dulac} function of $X$ in $\Omega$.
\end{definition}

In general,
it is very difficult to find a \pn{Dulac} function for some given vector field $X$,
but if one incidentally guesses such a function,
it is quite easy to verify that it obeys condition \eqref{eq:DulacCondition}%
\footnote{One can compare this problem to the decomposition of a natural number into its prime factors: 
If some natural number $n$ together with some product 
$p:=\prod_{i=1}^{N}{p_i^{d_i}}$ of powers of prime numbers is given,
it is quite easy to decide whether $p$ is the prime factorization of $n$, $p=n$,
but, from our current state of knowledge, it is very difficult to decompose a big natural number into its prime factors.}%
. Just for curiosity, the author wondered whether the converse statement of theorem \ref{theo:DulacCriterion} is also true.

\begin{qst}\label{qst:Dulac}
Given a  smooth planar vector field $X$
and assume that the planar differential equation \eqref{eq:planarODE} does not have any periodic solution in some simply connected domain $\Omega$.\\
Does there exist a \pn{Dulac} function in $\Omega$?
\end{qst}

The answer is No if the boundary of $\Omega$ is formed by a periodic solution of the corresponding system,
but in every other case considered by the author
he was able to obtain an affirmative answer to question \ref{qst:Dulac} raised above.
In the following we quote some of these positive results.

\section{Local results}\label{sec:localResults}
Gradient systems never possess periodic solutions \cite{GuckenheimerHolmes},
hence we expect that gradient fields have a \pn{Dulac} function.
Indeed, this is the case.
\begin{theorem}[Gradient fields \cite{Himmel:2009:OTE}]
Let $X=\nabla{V}$ be a globally defined gradient field with potential $V\in\cont^2(\R^2,\R)$.
Then $B_1=\exp{(V)}$, $B_2=\exp{(-V)}$ and $B_3=V$ are local \pn{Dulac} functions for $X$ in $\Omega_i \subset \R^2$, $i=1,2,3$,
with $\cup_{i=1}^3{\Omega_i}=\R^2$.
\end{theorem}

\begin{proof}
Define $B_1=\exp{(V)}$, $B_2=\exp{(-V)}$ and $B_3=V$ and
set $\Omega_i:=\{z\in\R^2|\, \Div(B_i X)(z)>0\}$.
Verify that $\cup_{i=1}^3{\Omega_i}=\R^2$ gives indeed the whole plane.
\end{proof}

A standard result from calculus says that vector fields can be straightened locally unless there is an equilibrium.
This fact can be used to obtain a local existence result in domains not containing zeros of the vector field.
Such domains are called canonical regions.
\begin{theorem}[Parallel flow \cite{MR2256001}]\label{theo:straighten}
Let $X:\D\subseteq \R^2 \to \R^2$ be any smooth vector field
and $\Omega$ a canonical region of $X$,
\ie a region where the flow of system \eqref{eq:planarODE} is equivalent to a parallel flow in the sense of Neumann \cite{MR0356138}.
%Moreover, assume that $\int_{\Omega}{\Div{X}}<\infty$.
Then, under some integrability assumptions, $X$ has a \pn{Dulac} function in $\Omega$.
%Note that there cannot be any zeros of $X$ in canonical regions.
\end{theorem}

\begin{proof}
By assumption there are no equilibria in the domain, hence the vector field can be straightened locally.
Observe that, by this process of straightening, the planar system \eqref{eq:planarODE} decouples
and the partial differential equation \eqref{eq:defDulacFunction} from the definition of a \pn{Dulac} function 
reduces to the linear ordinary differential equation
\begin{equation*}
B \Div{X}+\alpha B_r =g
%\label{eq:}
\end{equation*}
which implies
\begin{equation*}
B_r=\frac{g-B\Div{X}}{\alpha}
%\label{eq:}
\end{equation*}
This is an ordinary differential equation with solution
\begin{equation*}
B=\e^{a(r)} \Big\{\int{g \cdot \e^{a(r)}}{\,dr}+\kappa \Big\}
%\label{eq:}
\end{equation*}
where
$a(r)=\int{\Div{X}}{\, dr}$ and $\kappa \in\R$ is a constant of integration.
%defined as long as $\Div{X}$ and the right hand side $g$ are integrable.
\end{proof}

For linear vector fields there is also a \pn{Dulac} function in terms of a quadratic function
if some spectral condition holds.
\begin{theorem}[Linear case \cite{Himmel:2009:OTE}]
\label{theo:linearCase}
Let $X(z)=Az$ be a linear vector field with $A$ having at most one eigenvalue zero.
The linear system \eqref{eq:planarODE} does not have periodic orbits
if and only if
the spectrum of the matrix $A$ consists only of eigenvalues with nonzero real part or zero,
\ie \, $\sigma{(A)}\cap \ii\R \subseteq \{0\}$. 
\end{theorem}

\begin{proof}
Let the vector field be $X(z)=Az$ with 
\begin{equation*}
\begin{pmatrix}
a & b\\
c & d	
\end{pmatrix}
\in \R^{2\times 2}
%\label{eq:}
\end{equation*}
Instead of the general partial differential inequality \eqref{eq:DulacCondition} the author considered
\begin{equation}
\Div(B X)=\norm{X}^2:=P^2+Q^2
\label{eq:DulacNormSquare}
\end{equation}
and made a quadratic ansatz for the \pn{Dulac} function
\begin{equation*}
B=\frac{1}{2}\langle{z, Gz}\rangle
%\label{eq:}
\end{equation*}
with some matrix $G \in \R^{2 \times 2}$.
Then equation \eqref{eq:DulacNormSquare} reduces to 
\begin{equation*}
A^T G+ GA+\tr{A}\cdot G=A^T A.
%\label{eq:}
\end{equation*}
Letting $S:=A+\tr{A}$,
one obtains
\begin{equation}
S^TG+GS=A^T A.
\label{eq:Lyapunov}
\end{equation}
Equation \eqref{eq:Lyapunov} is the well-known \pn{Lyapunov} equation,
a special case of the more general \pn{Silvester} equation.
Now one can apply the well-established solvability theory for the \pn{Silvester} equation \cite{BhatiaRosenthal:1997:HAW}
and verify that the \pn{Lyapunov} equation \eqref{eq:Lyapunov} indeed 
has a unique solution under the spectral assumptions that were made.
On the other hand,
one obtains by direct calculation
\begin{equation*}
B=
b_{20}x^2+b_{02}y^2+b_{11}xy+b_{10}x+b_{01}y+b_{00}
\label{eq:squareDF}
\end{equation*}
with
\begin{equation*}
b_{20}=\frac{a^4+4a^3 d -a^2(2bc-c^2-3d^2)+3acd (c-b)+c^2(b^2-bc+d^2)}{(a+d)(3a^2+10ad-4bc+3d^2)},
%\label{eq:}
\end{equation*}
\begin{equation*}
b_{02}=\frac{a^2(b^2+3d^2)+ad (3b^2-3bc+4d^2)-b^3 c+b^2(c^2+d^2)-2bcd^2+d^4}{(a+d)(3a^2+10ad -4bc+3d^2)},
%\label{eq:}
\end{equation*}
\begin{equation*}
b_{11}=\frac{2a^3 b+a^2 d (7b+3c)-a(3b^2 c+b(c^-3d^2)-7cd^2)-cd(b^2+3bc-2d^2)}{(a+d)(3a^2+10ad -4bc+3d^2)},
%\label{eq:}
\end{equation*}
and $b_{01}=b_{10}=b_{00}=0$.
The case $\tr{A}=a+d=0$ has to be examined with care.
The reader may verify
that having spectrum $\sigma{(A)}=\{0, \frac{1}{2}\tr{A}\}$
is equivalent to
$\det{A}=\left(\frac{\tr{A}}{2}\right)^2$.
Hence the quadratic \pn{Dulac} function from \eqref{eq:squareDF} does the job
because we assumed that at most one eigenvalue of $A$ is zero. 
\end{proof}

The \pn{Hartman Grobman} theorem combined with the last two results
%\ref{theo:straighten} and \ref{theo:linearCase}
gives some local existence statement holding in some neighborhood of hyperbolic fixed points.
\begin{prop}\label{theo:DFnearHyperbolicFP}
Let $X$ be a smooth vector field and $z$ a hyperbolic zero
of $X$,
\ie, the real part $\re{z}\neq 0$ is different from zero.
Then there is a neighborhood $U$ of $z$ such that
$X$ has a \pn{Dulac} function on $U$. 
\end{prop}  
\begin{remark}
Proposition \ref{theo:DFnearHyperbolicFP} says morally that near to hyperbolic equilibriums one can define \pn{Dulac} functions,
which is what we expected since in dimension $n=2$ a hyperbolic equilibrium is either
a node (two real eigenvalues of the same sign),
a saddle (two real eigenvalues of different sign) or 
a focus, sometimes also called spiral point, (two complex conjugate eigenvalues with non-zero real part).
Can there be a \pn{Dulac} function near to a non-hyperbolic equilibrium?
We believe so, unless the equilibrium is a center.
Recall that a non-hyperbolic equilibrium (two purely imaginary eigenvalues of opposite sign)
can be either a center or a focus. 
\end{remark}
\section{Qualitative theory}
Let us recall now some basic definitions and results that are frequently used in the qualitative theory of planar differential equations.
For a more detailed introduction we refer to \cite{MR2737365} and \cite{MR2256001}. 
We call the system \eqref{eq:planarODE} integrable in the domain $\D$ if it has a first integral defined on this domain,
\ie
a non-constant smooth scalar-valued function
$H$ of class $\cont^k$ which is constant on each solution $(x(t),y(t))$ of \eqref{eq:planarODE} as long as it is defined.
This means:
if $(x(t),y(t))$ is any fixed solution of \eqref{eq:planarODE} defined for $t\in[0,t_{\text{max}}]=:I_{\text{max}}$, its maximal interval of existence,
and $H\in \cont^1(\D,\R)$ a first integral of system \eqref{eq:planarODE}, 
then there is a real number $h$ such that
\begin{equation}
H(x(t),y(t))=h \quad \text{ for all } t\in I_{\text{max}}
\label{eq:firstIntegral}
\end{equation} 
is satisfied.
Taking the derivative with respect to time $t$ of equation \eqref{eq:firstIntegral},
we see that any first integral $H\in \cont^1(\D,\R)$ of \eqref{eq:planarODE}
satisfies the linear partial differential equation
\begin{equation}
\langle{\nabla{H},
X}\rangle=P \cdot H_x+ Q \cdot H_y=0.
\label{eq:pdeFirstIntegral}
\end{equation}
Conversely,
every non-constant solution of equation \eqref{eq:pdeFirstIntegral} is a first integral $H:\D\to\R$ of \eqref{eq:planarODE}.
If $H_0$ is a non-constant solution of \eqref{eq:pdeFirstIntegral}, 
then every other solution is of the form $F(H_0)$ where $F$ is an arbitrary function having continuous partial derivatives (use the chain rule to verify this).
First integrals are strongly related to the notion of integrating factors.
\begin{definition}[Integrating factor]
An integrating factor $\mu$ of the planar dynamical system \eqref{eq:planarODE} is a smooth solution
$\mu \in \cont^1(\Omega,\R)$ of the linear partial differential equation
\begin{equation}
\Div(\mu \cdot X)=\mu\cdot \Div{X}+
\langle{\nabla{\mu},X}\rangle=0 \text{ in } \Omega.
\label{eq:defIntegratingFactor}
\end{equation}
Note that the first equality in \eqref{eq:defIntegratingFactor} is due to the \pn{Leibniz} rule
and $\Div{X}$, the divergence of the vector field $X$, denotes the trace of the Jacobian of $X$,
$\Div{X}=\frac{\partial{P}}{\partial{x}}+\frac{\partial{Q}}{\partial{y}}$.
\end{definition}
If instead a first integral $H: \D \to \R$ of \eqref{eq:planarODE} is known,
using equation \eqref{eq:pdeFirstIntegral},
one can equivalently define an integrating factor as the common value of the ratios 
\begin{equation}
\mu(x,y)=\frac{H_y}{P}\stackrel{!}{=}-\frac{H_x}{Q},
\label{eq:}
\end{equation}
\ie,
an integrating factor must satistfy both $H_y=\mu P$ and $H_x=-\mu Q$.
The latter two relations can be read as
\begin{equation}
d{H(x,y)}=\mu(x,y) (P(x,y) dy - Q(x,y) dx),
%\label{eq:}
\end{equation}
where $d{H(x,y)}$ indicates the differential of $H$.
Thus, multiplying the right hand side of \eqref{eq:planarODE} by an integrating factor makes
the equation an exact differential and an exact equation),
which means that, 
whenever an integrating factor is available,
the modified vector field
$\mu X$
has vanishing divergence and the problem of solving equation \eqref{eq:planarODE} is
reduced to 
one-dimensional integration
\begin{equation*}
H(x,y)=\int_{(x_0,y_0)}^{(x,y)}{\mu(x,y) (P(x,y)dy - Q(x,y)dx)}.
%\label{eq:}
\end{equation*}
Note that the latter line integral might not be well-defined if the domain $\D$ is not simply connected.
For this reason, 
integrating factors are usually considered only in connected components of $\D$.  
Secondly, we observe that the vector fields $X$ and $\mu X$ have the same phase portrait (with maybe reversed orientation if $\mu$ is negative)
as long as $\mu$ does not vanish.
That is why 
solving system \eqref{eq:planarODE},
\ie constructing a first integral,
and finding an integrating factor for it are considered to be equivalent problems.
In applications the notion of inverse integrating factor is very common.
%as it reflects already natural limitations on the domain of definition
\begin{definition}[Inverse integrating factor]
The function $V\in\cont^1(\Omega, \R)$ is called an inverse integrating factor for the planar system \eqref{eq:planarODE} in the domain $\Omega \subseteq R^2$  
if 
\begin{equation*}
\mu=\frac{1}{V}
%\label{eq:}
\end{equation*} is an integrating factor of system \eqref{eq:planarODE} in $\Omega \setminus \{V=0\}$.
As usual $\{V=0\}$ is short notation for
the preimage of zero under  $V$, $V^{-1}(0)=\{z\in\Omega:\, V(z)=0 \}$.
\end{definition}
The method of integrating factors is,
both historically and theoretically,
a very important technique in the qualitative analysis of first order ordinary differential equations.
The use of integrating factors goes back to \pn{Leonard Euler} (1707 - 1783).\\
Integration factors and first integrals refer to the problem of integrating the planar system \eqref{eq:planarODE},
which means geometrically nothing but finding smooth curves that are tangential to the vector field at each point.
Then, another interesting problem is deriving the qualitative behavior of these solution curves 
as, for instance, their topology (whether they are closed or not), 
their asymptotics (whether they blow up in finite time or remain in some compact set and approach a limit cycle or an equilibrium point)
and stability properties.

\section{A unifying point of view}
\begin{definition}[Invariant curve]
Let $\Omega \subseteq \R^2 $ be an open set of the plane.
An invariant curve is the vanishing set or the primage of zero of some smooth function.
More precisely, 
to a given function $f\in\cont^1(\Omega,\C)$
we associate 
the preimage of zero
\begin{equation}
f^{-1}(0)\equiv\{f=0\}:=\{z\in\Omega|\, f(z)=0\}
\label{eq:invariantCurve}
\end{equation}
and call it an invariant curve of the vector field $X=(P,Q)\in \cont^1(\D,\R^2)$
if there is a smooth function $k \in \cont^1(\Omega,\C)$, called cofactor of $f$, satisfying the relation 
\begin{equation}
\langle \nabla{f}, X\rangle=k \cdot f\quad \text{ for all } z \in \Omega
\label{eq:defInvariantCurve}
\end{equation}
or more explicitly
\begin{equation*}
f_x(z) \cdot P(z)+f_y(z) \cdot Q(z)=k(z) \cdot f(z).
%\label{eq:}
\end{equation*}
Here $\nabla{f}=(\frac{\partial{f}}{\partial{x}}, \frac{\partial{f}}{\partial{y}})^{T}$ denotes the gradient of $f$,
$\langle{\cdot,\cdot}\rangle$  is the canonical inner product of $\R^2$ and the subscripts of $f$ indicate partial derivates, $f_x=\frac{\partial{f}}{\partial{x}}$, $f_y=\frac{\partial{f}}{\partial{y}}$.
%Equation \eqref{eq:defInvariantCurve} must hold pointwise for all $z\in\Omega$.
Note that on the invariant curve $\{f=0\}$
the gradient of $f$, $\nabla{f}$, is orthogonal to the vector field $X$ by the defining property of invariant curves \eqref{eq:defInvariantCurve}.
By convention, the function $f$ defining the invariant curve $f^{-1}(0)\subseteq \R^2$ is called invariant function. 
\end{definition}
%
%In the following, we call both the cero preimage $f^{-1}(0)$ and the fucntion $f$ defining this preimage an inviant curve 
%- of the vector field $X$ in the region $\Omega$. 
%The latter two informations about the vector field involved and region where the invariant curve is defined will be oppressed if clear from the context.
For studying the multiplicity of invariant curves, the notion of exponential factors, 
a special case of invariant functions, is useful.
It allows the construction of first integrals for polynomial systems
via the same method used by \pn{Darboux}.
\begin{definition}[Exponential factor]
Given two $h,g \in \R[x,y]$ coprime polynomials, 
the function $\exp(\nicefrac{g}{h})$ is called an exponential factor for system \eqref{eq:planarODE} 
if there is a polynomial $k\in\R[x,y]$ of degree at most $d-1$, 
$d:=\max\{\deg{P}, \deg{Q}\}$ being the degree of the polynomial system \eqref{eq:planarODE},
satisfying the relation
\begin{equation}
\langle{\nabla{(\e^{\nicefrac{g}{h}})}, X}\rangle=k\cdot \e^{\nicefrac{g}{h}} 
%\label{eq:}
\end{equation}
\end{definition}
Note that obviously
$\{(\e^{\nicefrac{g}{h}})=0\}=\emptyset$,
but $\{g=0\}$ defines an invariant curve for system \eqref{eq:planarODE}.

\begin{definition}[Darboux function]
Any function of the form
\begin{equation}
\prod_{i=1}^{r}{f_i^{\lambda_i}}\prod_{j=1}^{l}{\big(\exp{\{\nicefrac{g_j}{h_j^{n_j}}}\}\big)^{\mu_j}}
\label{eq:darbouxFunction}
\end{equation}
where,
for $1\leq i\leq r$ and
$1\leq j \leq l$,
$f_i(z)=0$ and 
$g_j(z)=0$ 
are invariant curves for system \eqref{eq:planarODE},
$h_j$ is a polynomial of $\C[x,y]$, 
$\lambda_i$ and
$\mu_j$ 
are complex numbers
and 
$n_j$ is a natural number or zero,
is called Darboux function.

\end{definition}

The following remark reminds the reader of some facts about invariant curves.
\begin{remark}
\begin{enumerate}
	\item 
Invariant curves are very important in the qualitative study of dynamical systems
because they generalize the notion integrating factors and \pn{Dulac} functions.
Thus,
it is possible to interpret
integrating factors and \pn{Dulac} functions 
as invariant functions to certain cofactors: 
an integrating factor is nothing but an invariant function having cofactor 
\begin{equation}
k=-\Div{X}
\label{eq:integratingFactorCofactor}
\end{equation}
and
a \pn{Dulac} function is an invariant function with cofactor
\begin{equation}
k=-\Div{X}+ \frac{1}{f}\cdot g,
\label{eq:dulacFunctionCofactor}
\end{equation}
$g$ being a continuous function 
with $g>0$ for almost every  $z\in\Omega$.
Note that the latter interpretation of \pn{Dulac} function as an invariant function to a specific cofactor 
makes only sense when $f$ does not vanish,
\ie,
the \pn{Dulac} function is not defined \resp singular on the invariant curve, the vanishing set of the invariant function.
This already gives a clue on the natural boundaries on the maximal domain of definition of \pn{Dulac} functions.
\item
An easy observation is that,
if $f$ and $g$ are invariant functions with cofactor $k_f$ and $k_g$, 
\resp,
then their pointwise product $f \cdot g$ defines also an invariant curve
with cofactor $k_f+k_g$.
\item
\Wlog,
we will always consider complex-valued invariant functions
because,
if $f$ is an invariant function with cofactor $k$ 
(with respect to some vector field in some domain),
then its conjugate function $\bar{f}$ is also an invariant function having cofactor $\bar{k}$,
and  therefore
the product
$f\cdot\bar{f}$ is a real-valued invariant function
with cofactor $k+\bar{k}$.
The same holds for exponential factors.
\item
In the case of polynomial planar vector fields
%,which is subject of the unsolved $16$-th \pn{Hilbert} problem,
the algebraic part of invariant curves and exponential factors has already been developed.
We quote some of these results and refer to \cite{MR2166493}, \cite{MR1106193} and \cite{MR1258526} for further reading.
Therefore, 
let the vector field $X\in \R_{d}^2[x,y]$ be a polynomial vector field of degree $d$.
%\ie
%$P,Q \in \R[x,y]$ with $d=\max\{\deg{P}, \deg{Q}\}$,
%or shorter .
Furthermore, 
assume that 
$\frac{d(d+1)}{2}+1$
different irreducible invariant algebraic curves
are known.
Then one can construct
a first integral 
of the form
\begin{equation}
H=f_{1}^{\lambda_1}\cdot \ldots \cdot f_{s}^{\lambda_s} 
\label{eq:darbouxFirstIntegral}
\end{equation}
where each 
$f_i(x, y)$ defines an invariant
algebraic curve
\begin{equation}
f_i(x, y)=0
\label{eq:invariantAlgebraicCurve}
\end{equation}
for system \eqref{eq:planarODE}
and $\lambda_i\in\C$ not all of them null, for $i = 1, 2,\ldots, s$,
$s\in\N$.
The functions of type \eqref{eq:darbouxFunction}
are called
Darboux functions.
%(\see \cite{})
\item
The irreducibility of the invariant functions in the algebraic case
must be replaced by the condition
\begin{equation}
\{p\in\Omega:\, f(p)=0\text{ and } \nabla{f}(p)=0\} \subseteq \{p\in\Omega:\, X(p)=0\}
%\label{eq:}
\end{equation}
in the non-algebraic case.
\item 
An easy observation:
Any invariant curve $\{f=0\}$ has exactly one cofactor
\begin{equation*}
k=\frac{\langle{\nabla{f},X}\rangle}{f}.
%\label{eq:}
\end{equation*}
%defined only if $f$ does not vanish.
\end{enumerate}
\end{remark}

\section{Open questions}
In section \ref{sec:localResults}, proposition \ref{theo:DFnearHyperbolicFP}, we obtained a local existence result for \pn{Dulac} functions
near to a hyperbolic equilibrium.
The proof was based on theorem \ref{theo:linearCase},
where we calculated a \pn{Dulac} function explicitly in terms of a quadratic polynomial.
In this proof we observed, in fact, 
why one needs to impose that at most one eigenvalue of the matrix $A$ is zero,
because, if we had two such eigenvalues $\sigma(A)=\{0\}$, the matrix would have trace zero
and we would have divided by it. 
Then, by applying the \pn{Hartman Grobman} theorem, the result carried over to hyperbolic fixed points.
Observe that 
neither does the \pn{Hartman Grobman} theorem hold for equilibria $p$ with linearization
having purely imaginary or zero spectrum, 
nor could we make any use of it,
because in both cases $\Div{X}(p)=0$ holds and our \pn{Dulac} function blows up.
Hence, in future the
\begin{qst}
When does a \pn{Dulac} function exist near to non-hyperbolic fixed points?
\end{qst}
for nonlinear vector fields has to be addressed. 
This question is somehow naive because one has do deal here with the center-focus-problem,
that is still not completely solved in general.
But this \pn{Dulac} approach may give a new perspective on it.
%\cite{MR2737365}
Why do we bother about all these local existence statements for \pn{Dulac} functions?
In principle, the motivation arose from the following
\begin{alg}\label{alg:Dulac}
Input: a nonlinear planar vector field $X$;
Output: number and position of periodic orbits together with phase portrait.
\begin{enumerate}
	\item 
	\textbf{Step1:} Determine the zeros of $X$.
	\item
	\textbf{Step2:} At each zero, define locally a \pn{Dulac} function.
	\item
	\textbf{Step3:} Extend them as long as possible. If no further extension is possible,
	one has found the limit cycles and determined the phase portrait. 
\end{enumerate}
\end{alg}
Of course, this is rather a pseudo algorithm because one cannot accomplish any of its steps.
One more approachable but somehow technical step is to combine these local results to a global one.
Technically, this means, as one has to deal here with different local coordinate representations,
that one has to glue together different diffeomorphisms.
Summing up, 
we claim that the \pn{Dulac} method in the spirit of algoritm \ref{alg:Dulac}
will give new insights in the qualitative theory of planar differential equations.

% (see \ref{pic:gluingDiffeo})?

%\begin{center}
%\end{center}
%\iffalse

\begin{figure}
%[!htb]

%\centering
\begin{flushleft}
\begin{minipage}[b]{0.9\textwidth}

\includegraphics[width=1.4\linewidth,height=1.4\textheight,
keepaspectratio]{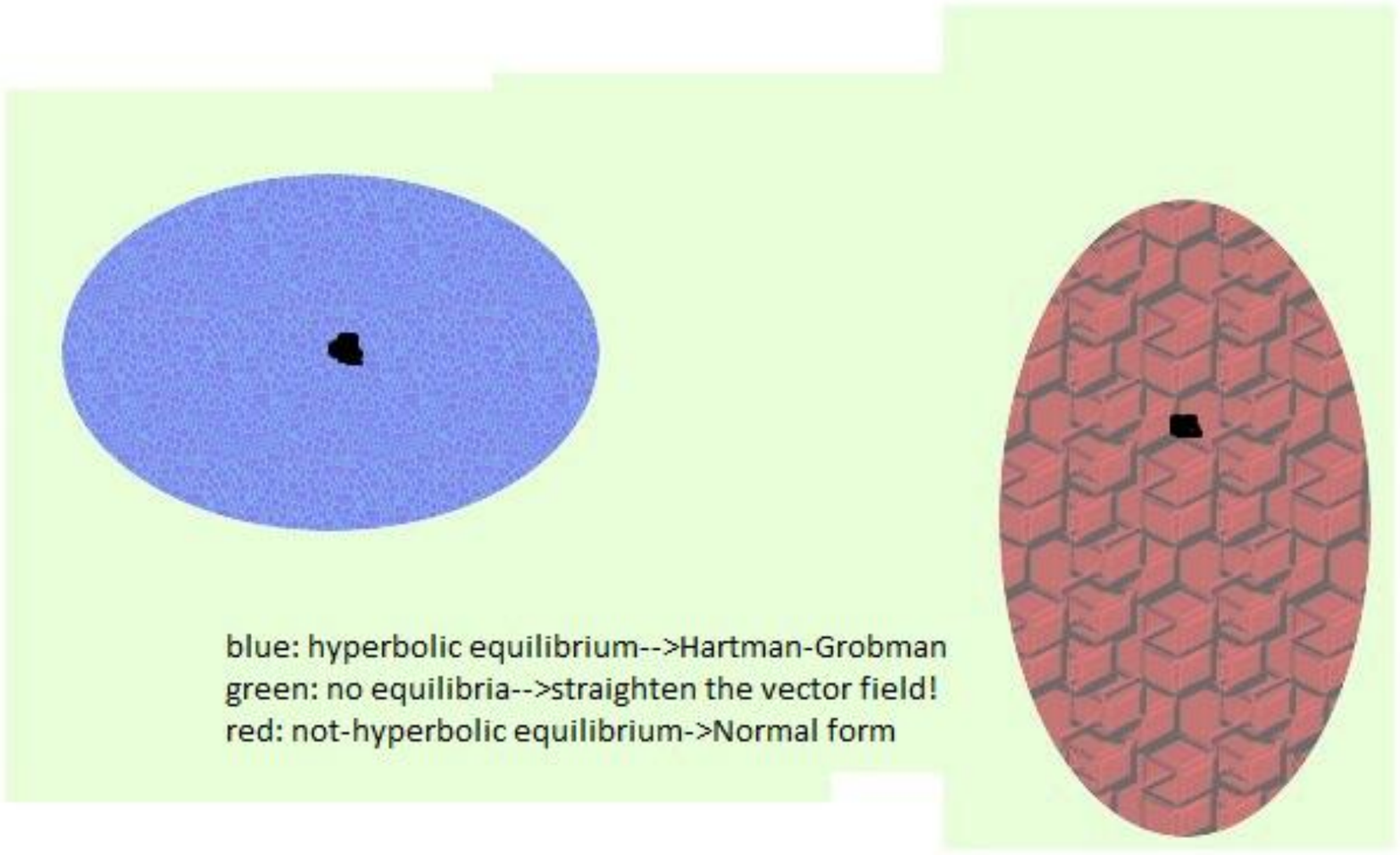}
  \caption{Gluing three different diffeomorphisms}
  \label{pic:gluingDiffeo}
  \end{minipage}
\end{flushleft}
\end{figure}
\newpage
%\fi

\bibliography{myrefs}{}
\bibliographystyle{plain}

\end{document}